\documentclass[11pt,reqno]{amsart}
\allowdisplaybreaks[4]
\RequirePackage[OT1]{fontenc}
\RequirePackage{amsthm,amsmath}
\RequirePackage[numbers]{natbib}
\RequirePackage[colorlinks,linkcolor=blue,citecolor=blue,urlcolor=blue]{hyperref}
\usepackage{amssymb}
\usepackage{anysize}
\usepackage{hyperref}
\usepackage{extarrows}
\usepackage{multirow}
\usepackage{natbib}
\usepackage{stmaryrd}
\usepackage{color}
\usepackage{arydshln}
\usepackage{listings}
\usepackage{comment}
\usepackage{pgf}
\usepackage{subfig}
\usepackage{yhmath}

\numberwithin{equation}{section}
\newtheorem{thm}{Theorem}[section]
\newtheorem{lem}[thm]{Lemma}
\newtheorem{pro}[thm]{Proposition}
\newtheorem{cor}[thm]{Corollary}
\theoremstyle{definition}

\theoremstyle{remark}
\newtheorem{rem}[thm]{Remark}

\theoremstyle{remark}

\newcommand{\be}{\begin{equation}}
\newcommand{\ee}{\end{equation}}
\newcommand{\bea}{\begin{eqnarray*}}
\newcommand{\eea}{\end{eqnarray*}}

\newtheorem{question}[thm]{Question}

\makeatletter

\newcommand{\Rmnum}[1]{\expandafter\@slowromancap\romannumeral #1@}
\makeatother
\makeatletter 
\@addtoreset{equation}{section}
\makeatother  

\newcommand{\abs}[1]{\left\lvert#1\right\rvert}

\newcommand{\beqa}{\begin{eqnarray}}
\newcommand{\eeqa}{\end{eqnarray}}

\newcommand{\eps}{\varepsilon}

\newcommand{\tr}{\mbox{tr\,}}

\renewcommand{\Re}{\mathsf{Re}}
\renewcommand{\Im}{\mathsf{Im}}

\newcommand{\E}{{\mathbb E }}

\renewcommand{\Pr}{{\mathbb P}}

\usepackage{todonotes}

\newcommand\tH{\widetilde H}
\newcommand\newbeta{\sqrt{\alpha}}
\newcommand\opnorm[1]{\left\|#1\right\|}
\newcommand\hsnorm[1]{\left\|#1\right\|_F}

\begin{document}

\title[Limiting ESD of the non-backtracking matrix for $G(n,p)$]{Limiting empirical spectral distribution for the non-backtracking matrix of an Erd\H os-R\'enyi random graph}

\author[K. Wang]{Ke Wang}
\address{Department of Mathematics, Hong Kong University of Science and Technology, Clear Water Bay, Kowloon, Hong Kong}
\email{kewang@ust.hk}

\author[P. M. Wood]{Philip Matchett Wood} 
\address{Department of Mathematics, Harvard University, 1 Oxford Street, Cambridge, MA 02138}
\email{pmwood@math.harvard.edu}

\thanks{Ke Wang was supported in part by Hong Kong RGC grants GRF 16308219, GRF 16304222, and ECS 26304920. Philip Matchett Wood was supported in part by NSA grant H98230-16-1-0301. }


\date{\today}

\begin{abstract} In this note, we give a precise description of the limiting empirical spectral distribution (ESD) for the non-backtracking matrices for an Erd\H{o}s-R\'{e}nyi graph $G(n,p)$ assuming $np/\log n$ tends to infinity.  We show that derandomizing part of the non-backtracking random matrix simplifies the spectrum considerably, and then we use Tao and Vu's replacement principle and the Bauer-Fike theorem to show that the partly derandomized spectrum is, in fact, very close to the original spectrum.  
\end{abstract}

\maketitle

\section{Introduction}\label{sec:intro}
For a simple undirected graph $G=(V,E)$, the non-backtracking matrix is defined as follows. For each $(i,j) \in E$, form two directed edges $i \to j$ and $j \to i$. The non-backtracking matrix $B$ is a $2|E| \times 2|E|$ matrix such that
$$B_{i\to j, k\to l}=
\begin{cases}
1 & \text{if}\ j=k ~\text{and}\ i\neq l \\
0 & \text{otherwise}.
\end{cases}
$$

The central question of the current paper is the following:

\begin{question}
What can be said about the eigenvalues of the non-backtracking matrix $B$ of random graphs as $|V|\to \infty$?
\end{question}

The non-backtracking matrix was proposed by Hashimoto \cite{H89}.
 The spectrum of the non-backtracking matrix for random graphs was studied by Angel, Friedman, and Hoory \cite{AFH} in the case where the underlying graph is the tree covering of a finite graph.  Motivated by the question of community detection (see \cite{Krz13,M2014,MNS13,MNS15}), Bordenave, Lelarge, and Massouli\'e \cite{BLM} determined the size of the largest eigenvalue and gave bounds for the sizes of all other eigenvalues for non-backtracking matrices when the underlying graph is drawn from a generalization of Erd\H os-R\'{e}nyi random graphs called the Stochastic Block Model (see \cite{HLL1983}), and this work was further extended to the Degree-Corrected Stochastic Block Model (see \cite{KN2011}) by Gulikers, Lelarge, and Massouli\'e \cite{GLM}. In  recent work,  Benaych-Georges, Bordenave and Knowles \cite{BBK17} studied the spectral radii of the sparse inhomogeneous Erd\H{o}s-R\'{e}nyi graph through a novel application of non-backtracking matrices.  Stephan and Massouli{\'e} \cite{SM20} also conducted a study on the non-backtracking spectra of weighted inhomogeneous random graphs.

In the current paper, we give a precise characterization of the limiting distribution of the eigenvalues for the non-backtracking matrix when the underlying graph is the Erd\H os-R\'{e}nyi random graph $G(n,p)$, where each edge $ij$ is present independently with probability $p$, and where we exclude loops (edges of the form $ii$).  We will allow $p$ to be constant or decreasing sublinearly with $n$, which contrasts to the bounds proved in \cite{BLM} and \cite{GLM} corresponding to the case $p=c/n$ with $c$ a constant.  Let $A$ be the adjacency matrix of $G(n,p)$, so $A_{ij}=1$ exactly when edge $ij$ is part of the graph $G$ and $A_{ij}=0$ otherwise; and let $D$ the diagonal matrix with $D_{ii}=\sum_{j=1}^n A_{ij}$.
Much is known about the eigenvalues of $A$, going back to works of Wigner in the 1950s 
\cite{Wigner1955,Wigner1958} (see also \cite{Grenader1963} and \cite{Arnold1967}), who proved that the distribution of eigenvalues follows the semicircular law  for any constant $p\in (0,1)$.  More recent results have considered the case where $p$ tends to zero, making the random graph sparse.  It is known that assuming $np\to \infty$, the empirical spectral distribution (ESD) of the adjacency matrix $A$ converges to the semicircle distribution (see for example \cite{KP93} or \cite{TVW13}). Actually, much stronger results have been proved about the eigenvalues of $A$ (see the surveys \cite{Vu2008} and \cite{BK16}). For example, Erd\H os, Knowles, Yau, and Yin \cite{EKYY} proved that as long as there is a constant $C$ so that $np > (\log n)^{C \log \log n}$ (and thus $np\to \infty$ faster than logarithmic speed), the eigenvalues of the adjacency matrix $A$ satisfy a result called the local semicircle law.  This law characterizes the distribution of the eigenvalues in small intervals that shrink as the size of the matrix $n$ increases. The most recent development regarding the local semicircle law can be found in \cite{HKM19} and \cite{ADK22}.

It has been shown in \cite{H, Bass, AFH} (for example, Theorem 1.5 from \cite{AFH}) that the spectrum of $B$ is the set $\{ \pm 1\} \cup \{\mu: \det(\mu^2 I - \mu A + D-I) =0\}$, 
or equivalently, the set $\{ \pm 1\} \cup \{\text{eigenvalues of } H \},$
where 
\begin{equation}\label{e:def_H}
H = \begin{pmatrix}
A & I-D\\
I & 0
\end{pmatrix}.
\end{equation}
We will call this $2n \times 2n$ matrix $H$ \emph{the non-backtracking spectrum operator for $A$}, and we will show that the spectrum of $H$ may be precisely described, thus giving a precise description of the eigenvalues of the non-backtracking matrix $B$.  We will study the eigenvalues of $H$ in two regions: the \emph{dense region}, where $p \in (0,1)$ and $p$ is fixed constant; and the \emph{sparse region}, where $p=o(1)$ and $np \to \infty$.  The \emph{diluted region}, where $p=c/n$ for some constant $c>1$, is the region for which the bounds in \cite{BLM} and \cite{GLM} apply, and, as pointed out by \cite{BLM}, it would be interesting to determine the limiting eigenvalue distribution  of $H$ in this region.




Note that $\E(D) = (n-1)p I$, and so we will let $$\alpha= (n-1)p-1$$ and consider the partly averaged matrix
\begin{equation}\label{e:def_H_0}
H_0= \begin{pmatrix}
A & I-\E(D)\\
I & 0
\end{pmatrix} = \begin{pmatrix}
A & -\alpha I\\
I & 0
\end{pmatrix}.
\end{equation}

The partly averaged matrix $H_0$ will be an essential tool in quantifying the eigenvalues of the non-backtracking spectrum operator $H$.  Three main ideas are at the core of this paper: first, that partial derandomization can greatly simplify the spectrum; second, that Tao and Vu's replacement principle \cite[Theorem~2.1]{TaoVu2010} can be usefully applied to two sequences of random matrices that are highly dependent on each other; and third, that in this case, the partly derandomized matrix may be viewed as a small perturbation of the original matrix, allowing one to apply results from perturbation theory like the Bauer-Fike Theorem.  The use of Tao and Vu's replacement principle here is novel,  as it is used to compare the spectra of a sequence of matrices with some dependencies among the entries to a sequence of matrices where all random entires are independent; 
typically, the Tao-Vu replacement principle has been applied in cases where the two sequences of random matrices both have independent entries, see for example \cite{TaoVu2010,Wood2012,Wood2016}.

\subsection{Results}

Throughout the rest of this paper, we mainly focus on sparse random graphs where $p\in (0,1)$ may tend to zero as $n\to \infty$.
Our first result shows that the spectrum of $H_0$ can be determined very precisely in terms of the spectrum of the random Hermitian matrix $A$, which is well-understood.

\begin{pro}[Spectrum of the partly averaged matrix]\label{t:H0spectrum}
Let $H_0$ be defined as in \eqref{e:def_H_0}, and let $0< p\le p_0<1$ for a constant $p_0$.  If $p\ge C/\sqrt{n}$ for some large constant $C>0$, then, with probability $1-o(1)$,  $\frac1\newbeta H_0$ has two real eigenvalues $\mu_1$ and $\mu_2$ satisfying $\mu_1=\newbeta (1+o(1))$ and $\mu_2 = 1/\sqrt{np} (1+o(1))$; all other eigenvalues for $\frac1\newbeta H_0$ are complex with magnitude $1$ and occur in complex conjugate pairs.   If $np \to\infty$ with $n$, then the real parts of the eigenvalues in the circular arcs are distributed according to the semicircle law.
\end{pro}  
\begin{rem}\label{rem:more} When $n^{-1+\epsilon}\le p\le n^{-1/2}$, more real eigenvalues of $H_0$ will emerge. We provide a short discussion on the real eigenvalues of $H_0$ in  Section \ref{sec:smallp}.  Note that as long as the number of real eigenvalues is bounded by a fixed constant,  for example when $p \ge C/\sqrt n$,   the bulk distribution of $H_0$ is two arcs on the unit circle,  with density so that real parts of the eigenvalues follow the semicircular law.
\end{rem}

The spectrum of the non-backtracking matrix for a degree regular graph was studied in \cite{B15}, including proving some precise eigenvalue estimates.  One can view Proposition~\ref{t:H0spectrum} as extending this general approach by using averaged degree counts, but allowing the graph to no longer be degree regular.  Thus, Proposition~\ref{t:H0spectrum} shows that partly averaging $H$ to get $H_0$ is enough to allow the spectrum to be computed very precisely.  Our main results are the theorems below, which show that the empirical spectral measures $\mu_H$ for $H$ and $\mu_{H_0}$ for $H_0$  are very close to each other, even for $p$ a decreasing function of $n$.  (The definitions of the measure $\mu_M$ for a matrix $M$ and the definition of almost sure convergence of measures are given in Section~\ref{ss:defs}).

\begin{thm}\label{t:bulkESD-intro}
Let $A$ be the adjacency matrix for an Erd\H os-R\'{e}nyi random graph $G(n,p)$. Assume $0< p\le p_0<1$ for a constant $p_0$ and $np/\log n \to \infty$ with $n$. Let $\frac 1 \newbeta H$ be a rescaling of the non-backtracking spectrum operator for $A$ defined in \eqref{e:def_H} with $\alpha = (n-1)p-1$, and let $\frac 1 \newbeta H_0$ be its partial derandomization, defined in \eqref{e:def_H_0}.  Then, $\mu_{\frac 1\newbeta H} - \mu_{\frac 1 \newbeta H_0}$ converges almost surely (thus, also in probability) to zero as $n$ goes to infinity.
\end{thm}

\begin{rem}When $p\gg \log n/n$, the graph $G(n,p)$ is almost a random regular graph and thus $H_0$ appears to be a good approximation of $H$. When $p$ becomes smaller, such an approximation is no longer accurate. In this sense, Theorem \ref{t:bulkESD-intro} is optimal.
\end{rem}

In Figure \ref{fig:compare}, we plot the eigenvalues of $\frac{1}{\sqrt{\alpha}}H$ and $\frac{1}{\sqrt{\alpha}}H_0$ for an Erd\H os-R\'{e}nyi random graph $G(n,p)$, where $n=500$. The blue circles mark the eigenvalues for $H/\sqrt{\alpha}$ and the red x's mark the eigenvalues for $H_0/\sqrt{\alpha}$. We can see that the empirical spectral measures of $H/\sqrt{\alpha}$ and  $H_0/\sqrt{\alpha}$ are very close for $p$ not too small.  As $p$ becomes smaller (note that here $\log n/n\approx0.0054$), the eigenvalues of $H_0/\sqrt{\alpha}$ still lie on the arcs of the unit circle whereas the eigenvalues of $H/\sqrt{\alpha}$ start to escape and be attracted to the inside of the circle. 

To prove that the bulk eigenvalue distributions converge in Theorem~\ref{t:bulkESD-intro}, we will use Tao and Vu's replacement principle \cite[Theorem~2.1]{TaoVu2010} (see also Theorem~\ref{t:Repl}), which was a key step in proving the circular law.  The replacement principle lets one compare eigenvalue distributions of two sequences of random matrices, and it has often been used in cases where one type of random input---for example, standard Gaussian normal entries---is replaced by a different type of random input---for example, arbitrary mean 0, variance 1 entries.  This is how the replacement principle was used to prove the circular law in \cite{TaoVu2010}, and it was used similarly in, for example, \cite{Wood2012, Wood2016}.  The application of the replacement principle in the current paper is distinct in that the entries for one ensemble of random matrices,  namely $H$,  has some dependencies between the entries,  whereas the random entries of $H_0$ are all independent.

Our third result (Theorem~\ref{thm:main-dense} below) proves that all eigenvalues of $H$ are close to those of $H_0$ with high probability when $p\gg \frac{\log^{2/3} n}{n^{1/6}}$, which implies that there are no outlier eigenvalues of $H$, that is, no eigenvalues of $H$ that are far outside the support of the spectrum of $H_0$ (described in Theorem \ref{t:H0spectrum}).

\begin{thm}\label{thm:main-dense}
Assume $0<p\le p_0<1$ for a constant $p_0$ and $p\ge \frac{\log^{2/3+\eps} n}{n^{1/6}}$ for $\eps>0$. Let $A$ be the adjacency matrix for an Erd\H os-R\'{e}nyi random graph $G(n,p)$. Let $\frac 1 \newbeta H$ be a rescaling of the non-backtracking spectrum operator for $A$ defined in \eqref{e:def_H}. Then, with probability $1-o(1)$, each eigenvalue of $\frac1 \newbeta H$ is within distance $R=
40\sqrt{\frac{\log n}{np^2}}$ of an eigenvalue of $\frac1 \newbeta H_0$, defined in \eqref{e:def_H_0}. 
%
\end{thm}

 In the upcoming Section \ref{sec:ESD}, it will be demonstrated that eigenvalues in the bulk of the distributions for $\frac1 \newbeta H$ and for $\frac1 \newbeta H_0$ have absolute value 1. Since $p\gg \frac{\log^{2/3} n}{n^{1/6}}$, we have $R=40\sqrt{\frac{\log n}{np^2}}=o(1)$, and consequently, Theorem 1.6 provides informative results.  We would like to mention that the above result has been improved to hold for $p \gg \log n/n$ and that each eigenvalue of of $\frac1 \newbeta H$ is within distance $
O((\frac{\log n}{np})^{1/4})$ of an eigenvalue of $\frac1 \newbeta H_0$, using a variant of Bauer-Fike perturbation theorem that appeared later in \cite[Corollary 2.4]{CZ21}, as opposed to invoking the classical Bauer-Fike theorem in this paper (see Theorem \ref{thm:bauer-fike}).  

\begin{figure}
 \centering\includegraphics[width=12cm]{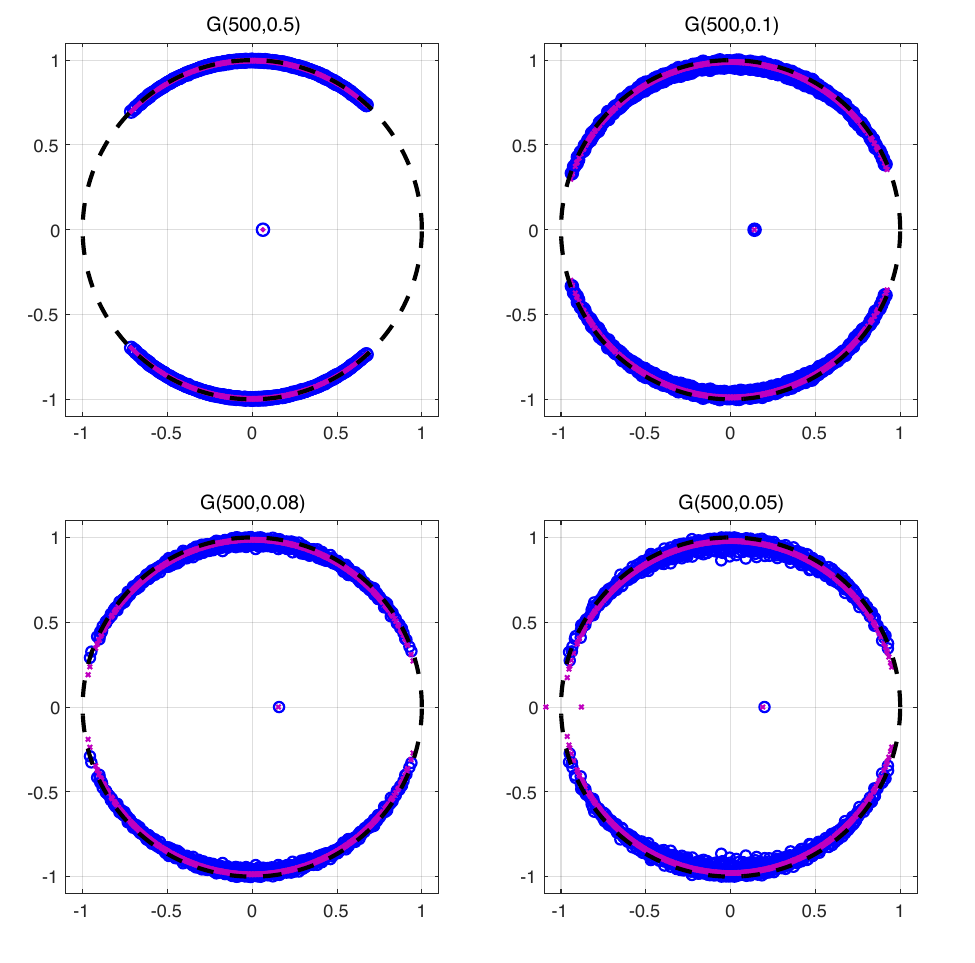}
 \caption{{\small The eigenvalues of $H/\sqrt{\alpha}$ defined in \eqref{e:def_H} and $H_0/\sqrt{\alpha}$ defined in \eqref{e:def_H_0} for a sample of $G(n,p)$ with $n=500$ and different values of $p$. The blue circles are the eigenvalues of $H/\sqrt{\alpha}$ and the red x's are for $H_0/\sqrt{\alpha}$. For comparison, the black dashed line is the unit circle. For the figures from top to bottom and from left to right, the values of $p$ are taken to be $p=0.5, p=0.1, p=0.08$ and $p=0.05$ respectively.}
}
 \label{fig:compare}
\end{figure}

\subsection{Outline}
We will describe the ESD of the partly averaged matrix $H_0$ to prove Proposition~\ref{t:H0spectrum} in Section~\ref{sec:ESD}. In Section~\ref{s:bulk}, we will show that the ESDs of $H$ and $H_0$ approach each other as $n$ goes to infinity by using the replacement principle \cite[Theorem~2.1]{TaoVu2010} and in Section \ref{sec:perturbation} we will use the Bauer-Fike theorem to prove Theorem~\ref{thm:main-dense}, showing that the partly averaged matrix $H_0$ has eigenvalues close to those of $H$ in the limit as $n\to \infty$.

\subsection{Background definitions}\label{ss:defs}
We give a few definitions to make clear the convergence described in Theorem~\ref{t:bulkESD-intro}  between empirical spectral distribution measures of $H$ and $H_0$.  For an $n\times n$  matrix $M_n$ with  eigenvalues $\lambda_1, \dots, \lambda_n$, the empirical spectral measure $\mu_{M_n}$ of $M_n$ is defined to be
$$\mu_{M_n} =\frac1n \sum_{i=1}^n \delta_{\lambda_i},$$
where $\delta_x$ is the Dirac delta function with mass 1 at $x$.   Note that $\mu_{M_n}$ is a probability measure on the complex numbers $\mathbb C$. 
The empirical spectral distribution (ESD) for $M_n$ is defined to be 
$$F^{M_n}(x,y)=\frac 1 n \#\left\{\lambda_i : \Re(\lambda_i) \le x \mbox{  and } \Im(\lambda_i) \le y  \right\}.$$	
 For $T$ a topological space (for example $\mathbb R$ or $\mathbb C$) and $\mathcal B$ its Borel $\sigma$-field, we can define convergence of a sequence $(\mu_n)_{n \ge 1}$ of random probability measures on $(T,\mathcal B)$ to a nonrandom probability measure $\mu$ also on $(T,\mathcal B)$ as follows.  We say that \emph{$\mu_n$ converges weakly to $\mu$ in probability as $n \to \infty$} (written $\mu_n \to \mu$ in probability) if for all bounded continuous functions $f: T \to \mathbb R$ and all $\epsilon >0$ we have
$$\Pr\left(\abs{\int_T f \, d\mu_n - \int_T f\, d\mu} > \epsilon\right) \to 0 \mbox{ as } n \to \infty.$$ 

Also, we say that \emph{$\mu_n$ converges weakly to $\mu$ almost surely as $n \to \infty$} (written $\mu_n \to \mu$ a.s.) if for all bounded continuous functions $f:T \to \mathbb R$, we have that $\abs{\int_T f \, d\mu_n - \int_T f\, d\mu} \to 0$ almost surely as $n \to \infty$.

 We will use $\|A\|_F := \tr( AA^*)^{1/2}$ to denote the Frobenius norm or Hilbert-Schmidt norm, and $\opnorm{A}$ to denote the operator norm. We denote $\|A\|_{\max}=\max_{ij}|a_{ij}|$.  We use the notation $o(1)$ to denote a small quantity that tends to zero as $n$ goes to infinity. We use the asymptotic notation $f(n)\ll g(n)$ if $f(n)/g(n)=o(1)$ and we write $f(n)= o(g(n))$ ; furthermore,  we write $f(n)=O(g(n))$ if $f(n)\le C g(n)$ for a constant $C>0$ when $n$ is sufficiently large.   Finally,  we will use $I$ or $I_n$ to denote the identity matrix,  where the subscript $n$ will be omitted when the dimension can be inferred by context.

\section{The spectrum of $H_0$}\label{sec:ESD}


We are interested in the limiting ESD of $H$ when $H$ is scaled to have bounded support (except for one outlier eigenvalue), and so we will work with the following rescaled conjugation of $H$, which has the same eigenvalues as $H/\sqrt{\alpha}$.
\begin{align*}
\widetilde{H}:=\frac{1}{\newbeta}\left(
\begin{array}{cccc}
\frac{1}{\newbeta}I   & 0\\
  0&I
\end{array}
\right) \left(
\begin{array}{cccc}
A   &I-D \\
I    &0
\end{array}
\right)   \left(
\begin{array}{cccc}
\newbeta I   & 0\\
 0 &I
\end{array}
\right)= \left(
\begin{array}{cccc}
\frac{1}{\newbeta}A  &  \frac1{\alpha}(I-D) \\
I   &0
\end{array}
\right).
\end{align*}

Note that the diagonal matrix $\frac1{\alpha}(I-D)$ is equal to $-I$ in expectation, and so we will compare the eigenvalues of $\tH$ to those of the partly averaged matrix $\tH_0$, noting that $\tH = \tH_0 + E$, where

\begin{equation}\label{e:def tH_0}
\tH_0 := \left(
\begin{matrix}
\frac1 \newbeta A & -I \\
I & 0
\end{matrix}
 \right)\qquad  \mbox{ and } \qquad
E:=
\left(
\begin{matrix}
0 & I+\frac1{\alpha}(I-D) \\
0 & 0
\end{matrix}
 \right).
 \end{equation}
Note that $H_0/\sqrt{\alpha}$ and $\tH_0$ also have identical eigenvalues. 

\newcommand\diag{\operatorname{diag}}

We will show that $\tH_0$ is explicitly diagonalizable in terms of the eigenvectors and eigenvalues of $\frac1 \newbeta A$, and then use this information to find an explicit form for the characteristic polynomial for $\tH_0$.

\subsection{Spectrum of $\tH_0$: Proof of Proposition \ref{t:H0spectrum}}\label{sec:H0}

Since $\frac 1 \newbeta A$ is a real symmetric matrix, it has a set 
$v_1,\ldots, v_n$ of orthonormal eigenvectors with corresponding real eigenvalues $\lambda_1 \ge \lambda_2 \ge \dots \ge \lambda_n$.  Thus we may write $\frac{A}{\sqrt \alpha} = U^{T} \diag(\lambda_1,\dots,\lambda_n) U$ where $U$ is an orthogonal matrix.  Consider the matrix $xI - \tH_0$, and note that 
\begin{align*}
 \begin{pmatrix}
I & 0\\
-xI & I
\end{pmatrix} (xI - \tH_0) =  \begin{pmatrix}
xI-\frac{1}{\sqrt \alpha} A & I\\
-x\big(xI-\frac{1}{\sqrt \alpha} A\big) -I & 0
\end{pmatrix},
\end{align*}
we see that $\det(xI-\tH_0) = \det(I + x(xI - \frac 1 \newbeta A) ) = \det(x^2 I - \frac x \newbeta A + I)$.  Conjugating to diagonalize $A$, we see that 
\begin{equation}\label{e:charpoly}
\det(xI-\tH_0) = \det(x^2 I -  x\diag(\lambda_1,\dots,\lambda_n) + I)
= \prod_{i=1}^n (x^2 -  \lambda_i x + 1).
\end{equation}
With the characteristic polynomial for $\tH_0$ factored into quadratics as in \eqref{e:charpoly}, we see that for each $\lambda_i$ of $\frac 1 \newbeta A$, there are two eigenvalues $\mu_{2i-1}$ and $\mu_{2i}$ for $\tH_0$ which are the two solutions to $x^2-\lambda_i x  +1=0$; thus,
\begin{equation}\label{e:evals}
\mu_{2i-1} = \frac{\lambda_i + \sqrt{ \lambda_i^2 - 4}}{2} \qquad \mbox{ and } \qquad 
\mu_{2i} = \frac{\lambda_i - \sqrt{ \lambda_i^2 - 4}}{2}.
\end{equation}

The eigenvalues of $A$ are well-understood. We use the following results that exist in literature.
\begin{thm}[\cite{KS03,LS18}]\label{thm:old}
Let $A$ be the adjacency matrix for an Erd\H os-R\'{e}nyi random graph $G(n,p)$. Assume $0< p\le p_0<1$ for a constant $p_0$ and $p\ge n^{-1+\phi}$ for a small constant $\phi>0$. Then for any $\epsilon>0$, the following holds with probability $1-o(1)$:
\begin{align*}
\lambda_1(A) = np(1+o(1));
\end{align*}
\begin{align*}
\max_{2\le i \le n}|\lambda_i(A)+p| \le L\sqrt{np(1-p)}+n^{\epsilon} \sqrt{np}\left(\frac{1}{(np)^2} + \frac{1}{n^{2/3}} \right),
\end{align*}
where $L=2+ \frac{s^{(4)}}{np} + O(\frac{1}{(np)^2})$ and $s^{(4)} =n^2p \left[ \frac{p^3 + (1-p)^3}{n^2 p(1-p)}-\frac{3}{n^2}\right].$
\end{thm}
\begin{proof}
We collect relevant results regarding the eigenvalues of $A$ from different works in the literature. In \cite{KS03}, it is shown that with probability $1-o(1)$, $\lambda_1(A) = (1+o(1))\max\{np, \sqrt{\Delta} \}$ where $\Delta$ is the maximum degree. As long as $np/\log n \to \infty$, $\max\{np, \sqrt{\Delta} \} = np$ (for the bounds on $\Delta$ see, for instance,  the proof of Lemma \ref{lem:Cond2} below). 

The operator norm of $A-\E A$ and the extreme eigenvalues of $A$ have been studied in various works (see \cite{FK1981,Vu07,BBK17,EKYY, LS18, HLY20, HK21}). In particular, in \cite[Theorem 2.9]{LS18}, assuming $p\ge n^{-1 + \phi}$, the authors proved that for any $\epsilon>0$ and $C>0$, the following estimate holds with probability at least $1-n^{-C}$: 
$$\left| \frac{1}{\sqrt{np(1-p)}} \| A - \E A \| -L \right| \le n^{\epsilon} \left(\frac{1}{(np)^2} + \frac{1}{n^{2/3}} \right)$$ with $L=2+ \frac{s^{(4)}}{np} + O(\frac{1}{(np)^2})$ and $s^{(4)} =n^2p \left[ \frac{p^3 + (1-p)^3}{n^2 p(1-p)}-\frac{3}{n^2}\right] = 1+ O(p).$ The conclusion of the theorem follows immediately from the classical Weyl's inequality that $\max_{2\le i \le n}|\lambda_i(A)+p|=\max_{2\le i \le n}|\lambda_i(A)-\lambda_i(\E A)|\le\|A-\E A\|$.
\end{proof}

Now we are ready to derive Proposition \ref{t:H0spectrum}. 
\begin{proof}[Proof of Proposition \ref{t:H0spectrum}]
Note that $\lambda_i=\lambda_i(A)/\sqrt{\alpha}$ and $\alpha=(n-1)p-1$. We have that 
\begin{align}\label{eq:eigNA}
\lambda_1=\sqrt{np}(1+o(1)) \quad \text{and}\quad \max_{2\le i \le n}|\lambda_i| \le 2\sqrt{1-p}(1+o(1))
\end{align}
with probability $1-o(1)$ by Theorem \ref{thm:old}. Therefore, for $\lambda_1$, we see from \eqref{e:evals} that $\mu_1, \mu_2$ are  real eigenvalues and 
\begin{align*}
\mu_1=\sqrt{np}(1+o(1)) \quad \text{and} \quad \mu_2=\frac{1}{\sqrt{np}}(1+o(1))
\end{align*}
with probability $1-o(1)$. 
Next, by Theorem \ref{thm:old}, it holds with probability $1-o(1)$  for any $2\le i \le n$ that
\begin{align*}
\lambda_i^2=\frac{\lambda_i^2(A)}{\alpha} \le \frac{1}{\alpha}\left[ L\sqrt{np(1-p)}+p+n^{\epsilon} \sqrt{np}\left(\frac{1}{(np)^2} + \frac{1}{n^{2/3}} \right)\right]^2.
\end{align*}
Since $p \ge C/\sqrt{n}$ for a sufficiently large constant $C$, we have
\begin{align*}
\lambda_i^2 &\le \frac{1}{\alpha} \left[ 2\sqrt{np(1-p)}+ O(\max\{p,(np)^{-1/2}, p^{1/2} n^{-1/6+\epsilon}\}) \right]^2\\
&=\frac{4 np (1-p) + O(\max\{p \sqrt{np},1, pn^{1/3 + \epsilon}\})}{np-(p+1)}\\
&=\frac{4 (1-p) + O(\max\{\sqrt{p/n},(np)^{-1}, n^{-2/3 + \epsilon}\})}{1-\frac{p+1}{np}} \\
&= 4(1-p)+O\left(\max\{\sqrt{p/n},(np)^{-1}, n^{-2/3 + \epsilon}\}\right) + O\left(\frac{1}{np}\right)\le 4-3p,
\end{align*}
for all sufficiently large $n$.
Hence, for all $i\ge 2$, we have $\lambda_i^2 <4$ and thus $\mu_{2i-1}, \mu_{2i}$ are complex eigenvalues with magnitude $1$ (since $|\mu_{2i-1}|=|\mu_{2i}|=1$). One should also note that $\mu_{2i-1}\mu_{2i} = 1$ for every $i$, and that whenever $\mu_{2i-1}$ is complex (i.e., $i\ge 2$), its complex conjugate is $\overline \mu_{2i-1}=\mu_{2i}$. 

Furthermore, note that $\Re \mu_{2i-1}=\Re \mu_{2i} =\lambda_i/2 = \lambda_i(A)/2\sqrt{\alpha}$. It is known that the empirical spectral measure of $A/\sqrt{np(1-p)}$ converges to the semicircular law supported on $[-2,2]$ assuming $np\to\infty$ (see for instance \cite{KP93} or \cite{TVW13}). We have the ESD of the scaled real parts of $\mu_j$
$$\frac{1}{2n}\sum_{j=1}^{2n} \delta_{\frac{2\Re\mu_{j}}{\sqrt{1-p}}} \to \mu_{sc}$$weakly almost surely where $\mu_{sc}$ is the semicircular law supported on $[-2,2]$. The proof of Proposition \ref{t:H0spectrum} is now complete. 
\end{proof}

\subsection{Real eigenvalues of $\tH_0$ when $p\le n^{-1/2}$}\label{sec:smallp} As mentioned in Remark \ref{rem:more}, when $p$ becomes smaller than $n^{-1/2}$, more real eigenvalues of $\tH_0$ will emerge. We can identify some of these eigenvalues, using recent results of \cite{EKYY, LS18, HLY20, HK21} in the study of the extreme eigenvalues of $A$.  For instance, in \cite[Corollary 2.13]{LS18}, assume $n^{2\phi-1} \le p \le n^{-2\phi'}$ for $\phi>1/6$ and $\phi'>0$. Then 
\begin{align}\label{eq:tw}
\lim_{n\to \infty}\Pr\left( n^{2/3} \Big(\frac{1}{\sqrt{np(1-p)}}\lambda_2(A) - \mathcal L -a \Big)\le s\right) = F_1^{TW}(s),
\end{align} where $\mathcal L= 2+\frac{1}{np} + O(\frac{1}{n^2p^2})$, $a=\sqrt{\frac{p}{n(1-p)}}$ and $F_1^{TW}(s)$ is the Tracy-Widom distribution function. Therefore, when $p\ge n^{-2/3+\epsilon}$, by noting that $F_1^{TW}(s) \to 1$ as $s\to \infty$ and selecting $s$ to be a large constant in \eqref{eq:tw}, we see that 
$$\lambda_2(A) = 2\sqrt{np(1-p)} + p + \sqrt{\frac{1-p}{np}} + O\left(\frac{\sqrt{np}}{n^{2/3}}\right).$$ Note that if $p<\frac{1-p}{n^{1/3}}$, then $p<\sqrt{\frac{1-p}{np}}$ and thus for $ n^{-2/3+\epsilon}\le p \le n^{-1/2}\ll n^{-1/3}$, 
\begin{align*}
\lambda_2^2 -4 &= \left(\frac{\lambda_2(A)}{\sqrt \alpha} \right)^2-4 = \frac{\left(2\sqrt{np(1-p)} + \sqrt{\frac{1-p}{np}} +p+ O\big(\frac{\sqrt{np}}{n^{2/3}} \big) \right)^2}{np-(p+1)}-4\\
&= \frac{\left(2\sqrt{np(1-p)} + \sqrt{\frac{1-p}{np}} +p \right)^2+O\left(\frac{np}{n^{2/3}} \right)}{np-(p+1)}-4\\
&= \frac{4np(1-p) + 4(1-p) + 4p\sqrt{np(1-p)}+O(\frac{np}{n^{2/3}})}{np-(p+1)}-4\\
&=\frac{4(1-p) +\frac{4(1-p)}{np} + 4p \sqrt{\frac{1-p}{np}}+O\left(\frac{1}{n^{2/3}} \right)}{1-\frac{p+1}{np}}-4\\
&=-4p + 4p \sqrt{\frac{1-p}{np}}+ \frac{4(1-p)(2+p)}{np} +O(n^{-2/3})>0.
\end{align*}
Hence, from \eqref{e:evals}, both $\mu_3$ and $\mu_4$ are real. The convergence result \eqref{eq:tw} holds for finitely many extreme eigenvalues of $A$ and thus they also generate real eigenvalues for $\widetilde H_0$.

The fluctuation of the extreme eigenvalues of $A$ has been obtained in \cite[Corollary 1.5]{HLY20} for $n^{-7/9} \ll p \ll n^{-2/3}$  and in \cite{HK21} for the remaining range of $p$ up to $p\ge n^{-1+\epsilon}$. One could use similar discussion as above to extract information about the real eigenvalues of $\widetilde H_0$. The details are omitted.

\subsection{$\tH_0$ is diagonalizable}

We can now demonstrate an explicit diagonalization for $\tH_0$.  Since $\mu_{2i-1}$ and $\mu_{2i}$ are solutions to $\mu^2-\mu\lambda_i +1=0$, one can check that the following vectors 
\begin{align}\label{e:lefteig}
y_{2i-1}^*=\frac{1}{\sqrt{1+|\mu_{2i-1}|^2}}\begin{pmatrix}
-\mu_{2i-1} v_i^T & v_i^T\\
\end{pmatrix} \quad\mbox{ and }\quad
y_{2i}^*=\frac{1}{\sqrt{1+|\mu_{2i}|^2}}\begin{pmatrix}
-\mu_{2i}v_i^T & v_i^T
\end{pmatrix}
\end{align}
satisfy $y_{2i-1}^* \tH_0 = \mu_{2i-1} y_{2i-1}^*$ and $y_{2i}^* \tH_0 = \mu_{2i} y_{2i}^*$ for all $i$.  Furthermore, $y_{2i-1}$ and $y_{2i}$ are unit vectors. 
For $1\le i\le n$, define the vectors 
\begin{align}\label{e:righteig}
x_{2i-1}=\frac{\sqrt{1+|\mu_{2i-1}|^2}}{\mu_{2i}-\mu_{2i-1}}\begin{pmatrix}
v_i \\
\mu_{2i} v_i
\end{pmatrix} \quad\text{ and }\quad
x_{2i}=\frac{\sqrt{1+|\mu_{2i}|^2}}{\mu_{2i-1}-\mu_{2i}}\begin{pmatrix}
v_i \\
\mu_{2i-1}v_i
\end{pmatrix}.
\end{align}
Defining
\begin{align*}
Y=\begin{pmatrix}
y_1^*\\
y_2^*\\
\vdots\\
y_{2n}^*
\end{pmatrix}
\quad \mbox{ and }\quad
X=\begin{pmatrix}
x_1, x_2, \ldots, x_{2n}
\end{pmatrix}
\end{align*}
we see that $X=Y^{-1}$ since $v_1,\ldots,v_n$ are orthonormal. Also it is easy to check that $Y\tH_0 X = \text{diag}(\mu_1,\ldots,\mu_{2n})$.

\section{The bulk distribution: proving Theorem~\ref{t:bulkESD-intro}}\label{s:bulk}

%

We begin by re-stating Theorem~\ref{t:bulkESD-intro} using the conjugated matrices defined in \eqref{e:def tH_0}.

\begin{thm}\label{t:bulkESD}
Let $A$ be the adjacency matrix for an Erd\H os-R\'{e}nyi random graph $G(n,p)$. Assume $0< p\le p_0<1$ for a constant $p_0$ and $np/\log n \to \infty$ with $n$. Let $\tH$ be the rescaled conjugation of the non-backtracking spectrum operator for $A$ defined in \eqref{e:def tH_0}, and let $\tH_0$ be its partial derandomization, also defined in \eqref{e:def tH_0}.  Then, $\mu_{\tH} - \mu_{\tH_0}$ converges almost surely (thus, also in probability) to zero as $n$ goes to infinity.
\end{thm}


To prove Theorem~\ref{t:bulkESD}, we will show that the bulk distribution of $\tH$ matches that of $\tH_0$ using the replacement principle \cite[Theorem~2.1]{TaoVu2010}, which we rephrase slightly as a perturbation result below (see Theorem~\ref{t:Repl}).  First, we give a few definitions that we will use throughout this section.
We say that a random variable $X_n \in \mathbb C$ is \emph{bounded in probability} if
$$\lim_{C\to\infty} \liminf_{n \to \infty} \Pr(\abs{X_n} \le C) = 1$$
and we say that $X_n$ is \emph{almost surely bounded} if
$$ \Pr\left( \limsup_{n\to\infty} \abs{ X_n} < \infty\right) = 1.$$

\newcommand\mym{m}

\begin{thm}[Replacement principle \cite{TaoVu2010}]\label{t:Repl}
Suppose for each $\mym$ that $M_\mym$ and $M_\mym+P_\mym$ are random $m \times m$ matrices with entries in the complex numbers.  Assume that
\begin{equation}\label{ReplCond1}
\frac 1\mym \hsnorm{ M_\mym }^2+ \frac 1\mym \hsnorm{M_\mym+P_\mym}^2 \mbox{ is bounded in probability (resp., almost surely)}
\end{equation}
and that, for almost all complex numbers $z \in \mathbb C$,
\begin{equation}\label{ReplCond2}
\frac1\mym \log \abs{\rule{0pt}{10pt}\det \left( M_\mym+P_\mym - zI\right) } - \frac1\mym \log \abs{\rule{0pt}{10pt} \det\left(M_\mym -zI\right)}
\end{equation}
converges in probability (resp., almost surely) to zero; in particular, this second condition requires that for almost all $z \in \mathbb C$, the matrices 
$M_\mym+P_\mym - zI$ and $M_\mym - zI$ have non-zero determinant with probability $1-o(1)$ (resp., almost surely non-zero for all but finitely many $\mym$). 

Then $\mu_{M_\mym} - \mu_{M_\mym+P_\mym}$ converges in probability (resp., almost surely) to zero. 
\end{thm}

Note that there is no independence assumption anywhere in Theorem~\ref{t:Repl}; thus, entries in $P_\mym$ may depend on entries in $M_\mym$ and vice versa.

We will use the following corollary of Theorem~\ref{t:Repl}, which essentially says that if the perturbation $P_\mym$ has largest singular value of order less than the smallest singular value for $M_\mym-zI$ for almost every $z\in \mathbb C$, then adding the perturbation $P_\mym$ does not appreciably change the bulk distribution of $M_\mym$.

\begin{cor}\label{t:ReplCor}
For each $\mym$, let $M_\mym$ and $P_\mym$ be random $m\times m$ matrices with entries in the complex numbers, and let $f(z,\mym)\ge 1$ be a real function depending on $z$ and $\mym$.  Assume that
\begin{equation}\label{RCorCond1}
\frac1\mym \|M_\mym\|_F^2 + \frac1\mym\|M_\mym+P_\mym\|_F^2 \mbox{ is bounded in probability (resp., almost surely)},
\end{equation}
and 
\begin{equation}\label{RCorCond2}
f(z,\mym) \|P_m\| \mbox{ converges in probability (resp., almost surely) to zero},
\end{equation}
and, for almost every complex number $z \in \mathbb C$,
\begin{equation}\label{RCorCond3}
\opnorm{\left(M_\mym-zI\right)^{-1}} \le f(z,\mym),
\end{equation} 
with probability tending to 1 (resp., almost surely for all but finitely many $\mym$).

Then $\mu_{M_\mym} - \mu_{M_\mym+P_\mym}$ converges in probability (resp., almost surely) to zero. 
\end{cor}

\begin{proof}
We will show that the three conditions \eqref{RCorCond1}, \eqref{RCorCond2}, and \eqref{RCorCond3} of Corollary~\ref{t:ReplCor} together imply the two conditions needed to apply Theorem~\ref{t:Repl}.  

First note that \eqref{RCorCond1} is identical to the first condition \eqref{ReplCond1} of Theorem~\ref{t:Repl}.  Next, we will show in the remainder of the proof that condition \eqref{ReplCond2} of Theorem~\ref{t:Repl} holds by noting that sufficiently small perturbations have a small effect on the singular values, and also the absolute value of the determinant is equal to the product of the singular values.

Let $z$ be a complex number for which \eqref{RCorCond3} holds, let $M_\mym-zI$ have singular values $\sigma_1 \ge \dots \ge \sigma_\mym$, and let $M_\mym+P_\mym-zI$ have singular values $\sigma_1+s_1 \ge \sigma_2+s_2 \ge \dots \ge \sigma_\mym+s_\mym$.  We will use the following result, which is sometimes called Weyl's perturbation theorem for singular values, to show that the $s_i$ are small.

\begin{lem}[{\cite[Theorem~1.3]{Cha09}}]
Let $A$ and $B$ be $m\times n$ real or complex matrices with singular values $\sigma_1(A)\ge \dots \ge \sigma_{\min\{m, n\}}(A) \ge 0$ and $\sigma_1(B)\ge \dots \ge \sigma_{\min\{m, n\}}(B) \ge 0$, respectively.  Then
$$ \max_{1\le j \le \min\{m, n\}} \abs{ \sigma_j(A) -\sigma_j(B) } \le \opnorm{A-B}.$$
\end{lem}
We then have that 
\begin{align*}
\max_{1\le i \le m}\abs{s_i} &\le \|P_m\|,
\end{align*}
and by \eqref{RCorCond3},
\begin{align*}
\max_{1\le i \le m}\frac{\abs{s_i}}{\sigma_i} \le f(z,m)\|P_m\|
\end{align*}
which converges to zero in probability (resp., almost surely) by \eqref{RCorCond2}. Thus we know that 
\begin{equation*}\label{log-error-bound}
\left| \log (1+s_i/\sigma_i )\right| \le 2\abs{ s_i/\sigma_i}\le 2f(z,m)\|P_m\|,
\end{equation*}
where the inequalities hold with probability tending to 1 (resp., almost surely for all sufficiently large $m$).
Using the fact that the absolute value of the determinant is the product of the singular values, we may write \eqref{ReplCond2} as
\begin{align*}
\left|\frac1\mym \left(\log \prod_{i=1}^\mym (\sigma_i+s_i)- \log\prod_{i=1}^m\sigma_i\right) \right|
&=
\frac1m \left| \sum_{i=1}^\mym \log\left( 1+\frac{s_i}{\sigma_i}\right)\right| \le 2f(z,m)\|P_m\|,
\end{align*} 
which converges to zero in probability (resp., almost surely) by \eqref{RCorCond2}.   Thus, we have shown that \eqref{ReplCond1} and \eqref{ReplCond2} hold, which completes the proof.
\end{proof}

\subsection{Proof of Theorem~\ref{t:bulkESD}}

The proof of Theorem~\ref{t:bulkESD} will follow from Corollary~\ref{t:ReplCor} combined with lemmas showing that the conditions \eqref{RCorCond1}, \eqref{RCorCond2}, and \eqref{RCorCond3} of Corollary~\ref{t:ReplCor} are satisfied.  Indeed, Lemma \ref{lem:Cond1} verifies \eqref{RCorCond1}, Lemma \ref{lem:Cond3} verifies  \eqref{RCorCond3} and \eqref{RCorCond2} follows by combining Lemma \ref{lem:Cond2} and Lemma \ref{lem:Cond3}. Note that the assumption $np/\log n \to \infty$ in Theorem~\ref{t:bulkESD} is only needed to prove conditions \eqref{RCorCond1} and \eqref{RCorCond2}.  Condition \eqref{RCorCond3} in fact follows for any $p$ and for more general matrices--see the proof of Lemma~\ref{lem:Cond3}.

In Corollary~\ref{t:ReplCor}, we will take $M_\mym$ to be the partly derandomized matrix $\tH_0$ and $P_\mym$ to be the matrix $E$ (see \eqref{e:def tH_0}), where we suppress the dependence of $\tH_0$ and $E$ on $n=\mym/2$ to simplify the notation.  There are two interesting features: first, the singular values of $\tH_0$ may be written out explicitly in terms of the eigenvalues of the Hermitian matrix $A$ (which are well understood; see Lemma~\ref{lem:Cond3}); and second, the matrix $E$ is completely determined by the matrix $\tH_0$, making this a novel application of the replacement principle (Theorem~\ref{t:Repl} and Corollary~\ref{t:ReplCor}) where the sequence of matrices $\tH_0 + E= \tH$ has some dependencies among the entries.

\begin{lem}\label{lem:Cond2}
Assume $0< p \le p_0<1$ for a constant $p_0$. Further assume $np/\log n\to \infty$.  For $E$ as defined in \eqref{e:def tH_0}, we have that $\|E\|\le 20 \sqrt{\frac{{\log n}}{np}}$ almost surely  for all but finitely many $n$. In particular, $\|E\|$ converges to zero almost surely for all but finitely many $n$.
\end{lem}

\begin{proof}
First, note that $\E D= (n-1)p I = (\alpha+1)I$ and thus $$E:=
\left(
\begin{matrix}
0 & I+\frac1{\alpha}(I-D) \\
0 & 0
\end{matrix}
 \right)= \left(
\begin{matrix}
0 & \frac1{\alpha}(\E D-D) \\
0 & 0
\end{matrix}
 \right).$$
Since $\E D-D$ is a diagonal matrix, it is easy to check that 
$$\|E\|=\|E\|_{\max} = \frac1{\alpha} \|D-\E D\|_{\max}=\frac1{\alpha}\max_{1\le i \le n} |D_{ii}- \E D_{ii}|=\frac1{\alpha}\max_{1\le i \le n} |D_{ii}- (n-1)p|.$$
Note that $D_{ii}$'s have the same distribution. By the union bound, it follows that for any $s>0$,
$$\Pr(\|E\| \ge s) \le n \Pr\left(\frac1{\alpha} |D_{11}- (n-1)p| \ge s \right)=n \Pr\left(\frac1{\alpha} \Big|\sum_{j=2}^n a_{1j}- (n-1)p \Big| \ge s \right).$$ 

Next we will apply the following general form of Chernoff bound. 
\begin{thm}[Chernoff bound \cite{Chernoff52}] Assume $\xi_1,\ldots,\xi_n$ are iid random variables and $\xi_i\in [0,1]$ for all $i$. Let $p=\E \xi_i$ and $S_n = \sum_{i=1}^n \xi_i$, then for any $\eps>0$,
\begin{align*}
&\Pr(S_n - np \ge n\eps) \le \exp\left(-\mbox{RE}(p+\eps || p) n \right);\\
&\Pr(S_n - np \le -n\eps) \le \exp\left(-\mbox{RE}(p-\eps || p) n \right)
\end{align*}
where $\mbox{RE}(p||q)= p\log(\frac{p}{q}) + (1-p) \log(\frac{1-p}{1-q})$ is the relative entropy or Kullback-Leibler divergence. 
\end{thm}
By our assumption, $np=\omega(n)\log n$ where $\omega(n)$ is a positive function that tends to infinity with $n$. Now take $K=(n-1)p+npt$ where $t=t(n)=10 \sqrt{\frac{\log n}{np}}$ (say).  Our assumption $np/\log n\to \infty$ implies $t\to 0$ with $n$. Thus
\begin{align*}
\Pr \left(\sum_{j=2}^n a_{1j} \ge K \right) =\Pr \left(\sum_{j=2}^n a_{1j} -(n-1)p \ge npt \right) \le \exp(-\mbox{RE}(p+pt|| p) n)
\end{align*}
where 
\begin{align*}
\mbox{RE}\left(p+pt ||p \right)&= p(1+t)\log(1+t) + (1-p-pt)\log\left(\frac{1-p-pt}{1-p} \right)\\
&=p(1+t)\log(1+t) - (1-p-pt )\log\left(1+\frac{pt}{1-p-pt} \right)\\
&> p(1+t)(t-t^2/2) - pt =pt^2(1-t)/2
\end{align*}
by the elementary inequalities $x-x^2/2<\log(1+x)< x$ for $x>0$.

Therefore, for $n$ sufficiently large, taking $t=5 \sqrt{\frac{\log n}{np}}$, we get
$$\Pr \left(\sum_{j=2}^n a_{1j} \ge (n-1)p+npt \right) \le \exp \left(-\frac{np t^2(1-t)}{2} \right)\le \exp(-10\log n)= n^{-10}.$$

Similarly, take $L=(n-1)p-npt$ where $t=5 \sqrt{\frac{\log n}{np}}$. Applying the Chernoff bound yields that
\begin{align*}
\Pr \left(\sum_{j=2}^n a_{1j} \le L \right) =\Pr \left(\sum_{j=1}^n a_{1j} -(n-1)p \le -npt \right) \le \exp\left(-\mbox{RE}(p-pt|| p) n \right).
\end{align*}
We take $n$ sufficiently large such that $t=t(n)<0.01$ (say). Then 
\begin{align*}
\mbox{RE}(p-pt ||p)&= p(1-t)\log(1-t) + (1-p+pt)\log\left(\frac{1-p+pt}{1-p} \right)\\
&=p(1-t)\log(1-t) - (1-p+pt )\log \left(1-\frac{pt}{1-p+pt} \right)\\
&> p(1-t)(-t-\frac{3}{5}t^2) + pt = \frac{1}{5}pt^2(2+3t) \ge \frac{2}{5}pt^2
\end{align*}
where we use the fact that $\log(1-x)<-x$ for $x\in(0,1)$ and $\log(1-x)>-x-\frac{3}{5}x^2$ for $x\in (0,0.01)$.
Hence, we get
\begin{align}\label{eq:dn}
\Pr\left(\sum_{j=2}^n a_{1j} \le (n-1)p-npt \right) \le \exp\left( - \frac{2}{5}pt^2\right)=\exp(-10\log n) = n^{-10}.
\end{align}
Since $2\alpha t = 2((n-1)p-1) t \ge npt$ for $n$ sufficiently large, it follows that
\begin{align*}
&\Pr\left(\|E\| \ge 10 \sqrt{\frac{\log n}{np}} \right) \le n\Pr \left(\left|\sum_{j=2}^n a_{1j}-(n-1)p \right| \ge 2\alpha t \right) \\
&\le n\Pr\left(\sum_{j=2}^n a_{1j}\ge (n-1)p + npt \right) +n\Pr\left(\sum_{j=2}^n a_{1j}\le (n-1)p - npt \right)
\le 2n^{-10}.
\end{align*}
By the Borel-Cantelli lemma, we have that $\|E\| \le 10 \sqrt{\frac{\log n}{np}}$ almost surely for all but finitely many $n$.

\end{proof}

To show \eqref{RCorCond1}, we combine Hoeffding's inequality and Lemma \ref{lem:Cond2} to prove the following lemma. 

\begin{lem}\label{lem:Cond1}
Assume $0< p \le p_0<1$ for a constant $p_0$. Further assume $np/\log n\to \infty$. For $\tH_0$ and $E$ as defined in \eqref{e:def tH_0}, we have that both 
$\frac 1{2n} \|\tH_0\|_F^2$ and  $\frac 1{2n} \|\tH_0+E\|_F^2$ are almost surely bounded. 
\end{lem}

\begin{proof}
We begin by stating Hoeffding's inequality \cite{Hoeffding1963}.
%

\newcommand\mybeta{\beta}
\begin{thm}[Hoeffding's inequality {\cite{Hoeffding1963}}]\label{thm:hoeffding} 
Let $\mybeta_1,\ldots,\mybeta_k$ be independent random variables such that for
$1\le i \le k$ we have
$\Pr(\mybeta_i \in [a_i,b_i]) = 1.$
Let $S:= \sum_{i=1}^k \mybeta_i$.  Then for any real $t$, 
$$\Pr(\abs{S- \E(S) }\ge kt ) \le 2\exp\left( - \frac{2 k^2 t^2}{\sum_{i=1}^k (b_i -a_i)^2}\right).$$
\end{thm}

Recall that $\alpha=(n-1)p-1$ and $\tH_0= \left(
\begin{matrix}
\frac1 \newbeta A & -I \\
I & 0
\end{matrix}
 \right)$, where $A=(a_{ij})_{1\le i, j \le n}$ is the adjacency matrix of an Erd\H os-R\'{e}nyi random graph $G(n,p)$.  Thus 
\begin{align*} 
\|\tH_0\|_F^2 &= \frac{1}{\alpha}\|A\|_F^2 + 2\|I\|_F^2 = \frac{1}{\alpha}\sum_{i,j} a_{ij}^2 + 2n =\frac{2}{\alpha}\sum_{i<j} a_{ij} + 2n.
\end{align*}  To apply Hoeffding's inequality, note that $a_{ij}$ $(i<j)$ are iid random variables each taking the value 1 with probability $p$ and 0 otherwise. Let $b_i=1$ and $a_i=0$ for all $i$, and let $k= {n \choose 2}$, which is the number of random entries in $A$ (recall that the diagonal of $A$ is all zeros by assumption).  Letting $S=\sum_{i<j} a_{ij}$, we see that $\mathbb E S = kp$ and so 
\begin{align*}
\Pr\left(\rule{0pt}{12pt}\abs{ S - kp } \ge kt\right) \le 2 \exp(-2kt^2).
\end{align*}
Since $\|\tH_0\|_F^2 = \frac{2}{\alpha} S + 2n$, we obtain that
\begin{align*}
\Pr\left(\rule{0pt}{12pt}\abs{ \frac{1}{2n}\|\tilde{H}_0\|_F^2 - 1 -\frac{kp}{n\alpha}} \ge \frac{kt}{n\alpha}\right) \le 2 \exp(-2kt^2).
\end{align*}
Take $t=p$. For $n$ sufficiently large, 
$$\frac{kt}{n\alpha} =\frac{n(n-1)p/2}{n[(n-1)p-1]}\le \frac{n(n-1)p/2}{n(n-1)p/2}=1$$ and since $p\ge \omega(n) \log n/n$ for $\omega(n)>0$ and $\omega(n)\to \infty$ with $n$, we get 
\begin{align*}
\Pr\left(\rule{0pt}{12pt} \frac{1}{2n}\|\tilde{H}_0\|_F^2  \ge 3\right) \le 2 \exp(-2kt^2)\le 2\exp(-\omega(n)^2 \log^2 n/2).
\end{align*}
By the Borel-Cantelli lemma, we conclude that $\frac{1}{2n}\|\tilde{H}_0\|_F^2$ is bounded almost surely. Since $\|E\|_{\max}=\|E\|$, by triangle inequality, we see
\begin{align*}
\frac 1{2n} \|\tH_0+E\|_F^2 &\le \frac{1}{2n} (\|\tilde{H}_0\|_F + \|E\|_F)^2\le \frac{1}{n} \|\tilde{H}_0\|_F^2 + \frac{1}{n} \|E\|_F^2 \\
&\le \frac{1}{n} \|\tilde{H}_0\|_F^2 + \|E\|.
\end{align*}
By Lemma \ref{lem:Cond2}, we get $\frac 1{2n} \|\tH_0+E\|_F^2$ is bounded almost surely. This completes the proof.

\end{proof}

The last part of proving Theorem~\ref{t:bulkESD} by way of Corollary~\ref{t:ReplCor} is proving that \eqref{RCorCond3} holds with $M_m=\tH_0$ and $f(z,m)=C_z$, a constant depending only on $z$.   The following lemma will be proved by writing a formula for the singular values of $\tH_0$ in terms of the eigenvalues of the adjacency matrix $A$, which are well understood.   A number of elementary technical details will be needed to prove that the smallest singular value is bounded away from zero, and these appear in Lemma~\ref{lem:sing_min}.

\begin{lem}\label{lem:Cond3}
Let $\tH_0$ be as defined in \eqref{e:def tH_0} and let $z$ be a complex number such that $\Im(z) \ne 0$ and $\abs{z} \ne 1$ (note that these conditions exclude a set of complex numbers of Lebesgue measure zero).  Then there exists a constant $C_z$ depending only on $z$ such that $\opnorm{ (\tH_0 -zI)^{-1}} \le C_z$ with probability 1 for all but finitely many $n$.
\end{lem}

\begin{proof}
We will compute all the singular values of $\tH_0 -zI$, showing that they are bounded away from zero by a constant depending on $z$.  The proof does not use randomness and depends only on facts about the determinant and singular values and on the structure of $\tH_0$; in fact, the proof is the same if $\tH_0$ is replaced with any matrix $\begin{pmatrix}
M& -I \\ I & 0
\end{pmatrix}
$ with $M$ Hermitian.

To find the singular values of $\tH_0$ we will compute the characteristic polynomial $\chi(\tilde w)$ for $(\tH_0 -zI)(\tH_0 -zI)^*$, using the definition of $\tH_0$ in \eqref{e:def tH_0}, and assuming that $\tilde w= w + 1+\abs{z}^2$; thus,
\begin{align*}
\chi(\tilde w)&:=\det\left((\tH_0 -zI)(\tH_0 -zI)^* - (w+1+\abs{z}^2)I\right) \\
&=
\det\begin{pmatrix}
\frac{A^2}{\alpha} - (z+\bar z) \frac A \newbeta - wI & \frac A\newbeta +(\bar z -z) I\\
\frac A\newbeta + (z - \bar z)I & -wI
\end{pmatrix}.
\end{align*}
We can use the fact that if $\begin{pmatrix}X & Y \\ Z & W\end{pmatrix}$ is a matrix composed of four $n\times n$ square blocks where $W$ and $Z$ commute, then $\det \begin{pmatrix}X & Y \\ Z & W\end{pmatrix} = \det(XW-YZ)$ (see \cite[Theorem 3]{Silvester}).  Thus, it is equivalent to consider
\begin{align*}
\det\left(
w\left(\frac{A^2}{\alpha} - (z+\bar z) \frac A \newbeta -w I\right) + \left(\frac A\newbeta +(\bar z -z) I\right)\left(\frac A\newbeta + (z - \bar z)I\right)
\right).
\end{align*}
Because $\frac{A}{\newbeta}$ is Hermitian, it can be diagonalized to $L=\diag(\lambda_1,\dots,\lambda_n)$, and thus the above determinant becomes:
\begin{align*}
&\det\left(
w\left(\frac{A^2}{\alpha} - (z+\bar z) \frac A \newbeta -w I\right) +  \left(\frac A\newbeta +(\bar z -z) I\right)\left(\frac A\newbeta + (z - \bar z)I\right)
\right)\\
&\quad= \det\left(
w\left(L^2 - (z+\bar z) L -w I\right) +  \left(L +(\bar z -z) I\right)\left( L + (z - \bar z)I\right)
\right)\\
&\quad= \prod_{i=1}^n
\left(w \left(\lambda_i^2 - (z+\bar z) \lambda_i -w \right)+ \left( \lambda_i + (z - \bar z)\right)\left(\lambda_i +(\bar z -z) \right)
\right)\\
&\quad= \prod_{i=1}^n \left(
-w^2
+ 
w\left( \lambda_i^2 - (z+\bar z) \lambda_i \right)
+
\lambda_i^2-(z-\bar z)^2
\right).
\end{align*}
The quadratic factors can then be explicitly factored, showing that each $\lambda_i$ generates two singular values for $\tH_0 -zI$, each being the positive square root of
$$
1+\abs{z}^2+\frac12\left( \lambda_i^2 - (z+\bar z)\lambda_i\right)
\pm
\frac12 \sqrt{ \left( \lambda_i^2 - (z+\bar z) \lambda_i \right)^2+4(\lambda_i^2-(z-\bar z)^2)  }.
$$

The proof of Lemma~\ref{lem:Cond3} is thus completed by Lemma~\ref{lem:sing_min} (stated and proved below), which shows that the quantity above is bounded from below by a positive constant depending only on $z$.
\end{proof}

\begin{lem}\label{lem:sing_min}
Let $z$ be a complex number satisfying $\Im(z) \ne 0$ and $\abs{z} \ne 1$.  Then for any real number $\lambda$, we have that
\begin{equation}\label{e:singsqr}
1+\abs{z}^2+\frac12\left( \lambda^2 - (z+\bar z)\lambda\right)
\pm
\frac12 \sqrt{ \left( \lambda^2 - (z+\bar z) \lambda \right)^2+4(\lambda^2-(z-\bar z)^2)  } \ge C_z,
\end{equation}
where $C_z$ is a positive real constant depending only on $z$.
\end{lem}

The proof of Lemma~\ref{lem:sing_min} is given in Appendix~\ref{ap:sing_min_pf} using elementary calculus, facts about matrices, and case analysis.  Lemma~\ref{lem:sing_min} completes the proof of Lemma~\ref{RCorCond3} and thus of Theorem~\ref{t:bulkESD}.

\section{Perturbation theory: proving Theorem \ref{thm:main-dense}}\label{sec:perturbation}

In this section, we study the eigenvalues of $H$ via perturbation theory. Recall from that the discussion in the beginning of Section \ref{sec:ESD} that $\tH$ in \eqref{e:def tH_0} has the same eigenvalues as $H/\sqrt{\alpha}$. We consider $\tH = \tH_0 + E$ where $E=
\left(
\begin{matrix}
0 & I+\frac1{\alpha}(I-D) \\
0 & 0
\end{matrix}
 \right).$  Note that $H_0/\sqrt{\alpha}$ and $\tH_0$ also have identical eigenvalues.

Let us begin by defining the spectral separation of matrices. Denote the eigenvalues of a matrix $M$ by $\eta_i(M)$'s. The spectral variation of $M+E$ with respect to $M$ is defined by
$$S_{M}(M +  E) = \max_{j} \min_i |\eta_j(M+E) - \eta_i (M)|.$$

\begin{thm}[Bauer-Fike theorem; see Theorem 6 from \cite{BLM}]\label{thm:bauer-fike}
If $H_0$ is diagonalizable by the matrix $Y$, then
$$S_{H_0}(H_0 +  E) \le \opnorm{E} \cdot \opnorm{Y} \cdot \|Y^{-1}\|.$$
Denote by $\mathcal{C}_i:=\mathcal{B}(\mu_i(H_0), R)$ the ball in $\mathbb{C}$ centered at $\mu_i(H_0)$ with radius $R=\opnorm{E} \cdot \opnorm{Y} \cdot \|Y^{-1}\|$. Let $\mathcal{I}$ be a set of indices such that $$(\cup_{i\in \mathcal{I}} \mathcal{C}_i) \cap (\cup_{i\notin \mathcal{I}} \mathcal{C}_i) = \emptyset.$$ Then the number of eigenvalues of $H_0 + E$ in $\cup_{i\in \mathcal{I}} \mathcal{C}_i$ is exactly $|\mathcal{I}|$.
\end{thm}

We will bound the operator norm of $E$ and the condition number $\opnorm{Y} \opnorm{Y^{-1}}$ of $Y$ to prove Theorem~\ref{thm:main-dense}.

By Lemma~\ref{lem:Cond2}, we know that $\opnorm{E} \le 20\sqrt{\frac{\log n}{np}}$ with probability 1 for all but finitely many $n$.


To bound the condition number of $Y$, we note that the square of the condition number of $Y$ is equal to the largest eigenvalue of $YY^*$ divided by the smallest eigenvalue of $YY^*$.  Using the explicit definition of $Y$ from \eqref{e:lefteig}, we see from the fact that the $v_i$ are orthonormal that 
\begin{align*}
YY^*= \diag(Y_1,\ldots,Y_n)
\end{align*}
where $Y_i$'s are $2\times 2$ block matrices of the following form 
\begin{align*}
Y_i=\begin{pmatrix}
y_{2i-1}^* y_{2i-1} & y_{2i-1}^* y_{2i}\\
y_{2i}^* y_{2i-1} & y_{2i}^* y_{2i}
\end{pmatrix}.
\end{align*}
Recall that
\begin{align*}
y_{2i-1}^*=\frac{1}{\sqrt{1+|\mu_{2i-1}|^2}}\begin{pmatrix}
-\mu_{2i-1} v_i^T & v_i^T\\
\end{pmatrix} \quad\mbox{ and }\quad
y_{2i}^*=\frac{1}{\sqrt{1+|\mu_{2i}|^2}}\begin{pmatrix}
-\mu_{2i}v_i^T & v_i^T
\end{pmatrix}
\end{align*}
We then have 
$
Y_i=\begin{pmatrix}
1 & \gamma_i\\
\overline{\gamma}_i & 1
\end{pmatrix}
$
where $$\gamma_i:= \frac{\mu_{2i-1}\overline{\mu}_{2i}+1}{\sqrt{(1+|\mu_{2i-1}|^2)(1+|\mu_{2i}|^2)}}.$$ It is easy to check that the eigenvalues of $Y_i$ are $1\pm |\gamma_i|$. The eigenvalues of $YY^*$ are the union of all the eigenvalues of the blocks, and so we will compute the eigenvalues $1\pm |\gamma_i|$ based on whether $\lambda_i$ produced real or complex eigenvalues for $\tH_0$.  

For $i=1$, the eigenvalue $\lambda_1$ produces two real eigenvalues for $\tH_0$.    Using the facts that $\mu_{1}\mu_{2}=1$ and $\mu_{1}+ \mu_{2}=\lambda_1$, which together imply that $\mu_{1}^2+\mu_{2}^2 = \lambda_1^2-2$, we see that in this case $\gamma_1^2=\frac{4}{\lambda_1^2}$, and so the two eigenvalues corresponding to this block are $1\pm |\gamma_1|=1 \pm 2/\abs{\lambda_1}$. By \eqref{eq:eigNA}, we see that $1\pm |\gamma_i| = 1\pm \frac{2}{\sqrt{np}}(1+o(1))$ with probability $1-o(1)$.

For $i\ge 2$, the eigenvalue $\lambda_i$ produces two complex eigenvalues for $\tH_0$, both with absolute value 1 (see Section~\ref{sec:ESD}).  In this case, $\gamma_i = \frac{1+\mu_{2i-1}^2}{2}$.  Again using the facts that $\mu_{2i-1}\mu_{2i}=1$ and $\mu_{2i-1}^2+\mu_{2i}^2 = \lambda_i^2-2$, we see that $\overline \gamma_i \gamma_i = \lambda_i^2/4$, which shows that the two eigenvalues corresponding to this block are $1 \pm \abs{\lambda_i}/{2}$.

By \cite{Vu07} (see Theorem \ref{thm:old} in Section~\ref{sec:ESD}) we know that when $p\ge \frac{\log^{2/3+\eps} n}{n^{1/6}}$,  $\max_{2\le i\le n} \abs{\lambda_i} \le 2 \sqrt{1-p} + O(n^{1/4}\log n/\sqrt{np})$ with probability tending to 1, and thus the largest and smallest eigenvalues coming from any of the blocks corresponding to $i\ge 2$ are $1 + \sqrt{1-p} +O(n^{1/4}\log n/\sqrt{np}) $ and $1-\sqrt{1-p}+ O(n^{1/4}\log n/\sqrt{np})$ with probability tending to 1.  Combining this information with the previous paragraph, we see that the condition number for $Y$ is 
\begin{align*}
&\sqrt{\frac{1+\sqrt{1-p}+O(n^{1/4}\log n/\sqrt{np})}{1-\sqrt{1-p}+O(n^{1/4}\log n/\sqrt{np})}}=\sqrt{\frac{(1+\sqrt{1-p})^2 + O(n^{-1/4}p^{-1/2}\log n)}{p+ O(n^{-1/4}p^{-1/2}\log n)}}\\
&=\sqrt{\frac{2}{p} \frac{1+\sqrt{1-p} -p/2 +O(n^{-1/4}p^{-1/2}\log n)}{1+O(n^{-1/4}p^{-3/2}\log n)}} \\
& = \sqrt{\frac{2}{p} \left( (1+\sqrt{1-p}) + O(n^{-1/4}p^{-3/2}\log n) \right)}
\le \frac{2}{\sqrt{p}}
\end{align*}
for $n$ sufficiently large. In the third equation above, we use the Taylor expansion. In the last inequality,  we use that $n^{-1/4}p^{-3/2}\log n \le \log^{-3\eps/2} n=o(1)$ since $p\ge \frac{\log^{2/3+\eps} n}{n^{1/6}}$. 

Finally, we apply Lemma \ref{lem:Cond2} and Bauer-Fike Theorem (Theorem~\ref{thm:bauer-fike}) with $R=\frac{40}{\sqrt{p}}\sqrt{\frac{\log n}{np}}=40\sqrt{\frac{\log n}{np^2}}$ to complete the proof.

\appendix

\section{Proof of Lemma~\ref{lem:sing_min}}\label{ap:sing_min_pf}

\newcommand\partder[1]{\frac{\partial}{\partial #1}}
\newcommand\secpartder[1]{\frac{\partial^2}{\partial #1^2}}

It is sufficient to show that the left-hand side of \eqref{e:singsqr} (replacing $\pm$ with $-$) is bounded below by a positive constant $C_z>0$ depending only on $z$.  Substituting $z=a+ib$ where $a,b\in \mathbb R$, we see that the left-hand side of \eqref{e:singsqr} is bounded below by   

\begin{align}
1+a^2{+}&\; b^2+\frac12\left( \lambda^2 - 2a\lambda\right)
-
\frac12 \sqrt{ \left( \lambda^2 - 2a \lambda \right)^2+4(\lambda^2+ 4b^2)  }\nonumber\\
&=1+a^2+ b^2+\frac{\lambda^2}2 -a\lambda
-
\sqrt{ \left( \frac{\lambda^2 - 2a\lambda }{2} \right)^2+\lambda^2+ 4b^2}
\label{e:singsqrab}
%
%
\end{align}
Note that the quantity in \eqref{e:singsqrab} is always at least zero because it is a singular value for a matrix.

Define a function $g$ by 
\begin{equation}\label{e:gdef}
 g(\lambda, a,\gamma) = 
1+a^2+ \gamma+\frac{\lambda^2}2 -a\lambda
-
\sqrt{ \left( \frac{\lambda^2 - 2a\lambda }{2} \right)^2+\lambda^2+ 4\gamma}
\end{equation}
(where we replaced $b^2$ by $\gamma$ in the quantity \eqref{e:singsqrab}).  It is sufficient to complete the proof if we show that for all real $a$ and $\gamma$ satisfying $\gamma > 0$ and $a^2+\gamma \ne 1$ that  $g(\lambda,a,\gamma) \ge C_{a,\gamma} > 0$, where $C_{a,\gamma}$ is a constant depending only on $a$ and $\gamma$.  We will prove this by considering three cases and using calculus.

\textbf{Case I:} 
$ \left(\frac{\lambda^2 - 2a\lambda }{2} \right)^2+\lambda^2 \ge 4$.
In this case we will minimize $g$ over $\gamma >0$.  Note that
$$\frac{\partial}{\partial \gamma} g = 1 - \frac{2}{\sqrt{ \left(\frac{\lambda^2 -2a\lambda}{2}\right)^2 + \lambda^2 + 4 \gamma}}.$$
The quantity in the denominator of the fraction above is always greater than 2 since $ \left(\frac{\lambda^2 - 2a\lambda }{2} \right)^2+\lambda^2 \ge 4$ and $\gamma >0$ by assumption; thus the derivative $\partder{\gamma} g$ is positive for all $\gamma$, and hence $g$ (viewed as a function of $\gamma$) is strictly increasing as $\gamma$ increases, for any $a$ and any $|\lambda| \ge 2$.   Also, note that
$$\secpartder{\gamma} g = \frac{4}{\left(\left( \frac{\lambda^2 - 2a\lambda }{2} \right)^2+\lambda^2+ 4\gamma\right)^{3/2}},$$
which is always strictly positive, showing that $g$ is concave upwards as a function of $\gamma$, for any $a$ and any $|\lambda| \ge 2$.
By the above we know that $\partder{\gamma} g(\lambda,a,\gamma/2) \ge 1-\frac{2}{\sqrt{4 + 4\gamma}}$, and we also know that $g(\lambda,a,\gamma) \ge 0$ for all $\lambda, a,\gamma$ (because it is a singular value); thus, using the fact that the tangent line at $\gamma/2$ is strictly below $g$ (since it is concave upwards), we see that $g(\lambda,a,\gamma) \ge \left(1-\frac{2}{\sqrt{4 + 4\gamma}}\right)\frac{\gamma}{2},$ which is a positive constant depending only on $\gamma >0$, completing the proof in Case I.

\textbf{Case II:} 
$ \left(\frac{\lambda^2 - 2a\lambda }{2} \right)^2+\lambda^2 < 4$ (which implies $\abs\lambda < 2$)
and $\abs{ 1 -\frac{\lambda^2}{4} -\gamma} \ge \frac12\abs{a^2+\gamma -1}$.
In this case, we will minimize $g$ over $a$.  Recalling the definition of $g$ in \eqref{e:gdef}, we see that

\begin{align*}
\partder{a}g &= 2a - \lambda - \frac{2\left( \frac{\lambda^2 - 2a\lambda }{2} \right)(-\lambda)}{2\sqrt{ \left( \frac{\lambda^2 - 2a\lambda }{2} \right)^2+\lambda^2+ 4\gamma}
}\\
&=(2a-\lambda)\left( 1 - \frac{\lambda^2}{2\sqrt{ \left( \frac{\lambda^2 - 2a\lambda }{2} \right)^2+\lambda^2+ 4\gamma}}\right).
\end{align*}

Because $\abs \lambda < 2$, we see that $\frac{\lambda^2}{2\sqrt{ \left( \frac{\lambda^2 - 2a\lambda }{2} \right)^2+\lambda^2+ 4\gamma}} < \frac{\abs\lambda}{\sqrt{ \left( \frac{\lambda^2 - 2a\lambda }{2} \right)^2+\lambda^2+ 4\gamma}} \le 1$, which shows that $\partder{a} g$ has the same sign as $2a-\lambda$ and thus achieves a minimum when $a = \lambda/2$.  Thus, using the minimum of $g$ over $a$ in this case, we have 
$$g(\lambda,a,\gamma) \ge g(\lambda,\lambda/2,\gamma) = 1 + \frac{\lambda^2}{4} + \gamma - \sqrt{\lambda^2 + 4 \gamma}.$$
Substituting $x = \sqrt{\frac{\lambda^2}4+\gamma}$ into the right-hand side, we see that $g$ is bounded below by the quantity $1+x^2 - 2x =(x-1)^2$ for positive $x$ satisfying $\abs{x^2-1} \ge \frac12\abs{a^2+\gamma -1}$ (which follows from the assumptions on $a$ and $\gamma$ in this case).  We will complete the argument by proving the claim that $\abs{x-1} \ge \min\{1, \frac16 \abs{a^2+\gamma -1} \}$, which is a positive constant depending only on $a$ and $\gamma$.  To prove this claim, note that if $\abs x \ge 2$, then $\abs{x-1}\ge 1$; and on the other hand, if $\abs x < 2$, then $\abs{x-1}=\frac{\abs{x^2-1}}{\abs{x+1}} > \frac13\abs{x^2-1} \ge \frac16\abs{a^2+\gamma -1}$. This completes the proof in Case II.


\textbf{Case III:} $ \left(\frac{\lambda^2 - 2a\lambda }{2} \right)^2+\lambda^2 < 4$ and $\abs{ 1 -\frac{\lambda^2}{4} -\gamma} < \frac12\abs{a^2+\gamma -1}$.  In this case, we will first minimize over $\gamma$.  Note that
$$\partder{\gamma} g = 1 - \frac{2}{\sqrt{ \left( \frac{\lambda^2 - 2a\lambda }{2} \right)^2+\lambda^2+ 4\gamma}
}.$$
Because $\left( \frac{\lambda^2 - 2a\lambda }{2} \right)^2+\lambda^2 < 4$, we see from $\partder{\gamma}g$ that $g$ has a minimum when $\gamma = \gamma_0(\lambda,a):=1-\frac{\lambda^2}{4}-\left( \frac{\lambda^2 - 2a\lambda }{4} \right)^2$.  Thus, $g(\lambda,a,\gamma) \ge g(\lambda,a,\gamma_0(a,\lambda))$.  Thus, we can find a general lower bound on $g$ in this case by demonstrating a lower bound on 
\begin{align}
g(\lambda,a,\gamma_0(a,\lambda))&= 1+a^2+1-\frac{\lambda^2}{4}-\left( \frac{\lambda^2 - 2a\lambda }{4} \right)^2+\frac{\lambda^2}{2}-a\lambda - 2\nonumber\\
&= a^2+\frac{\lambda^2}{4}-\left( \frac{\lambda^2 - 2a\lambda }{4} \right)^2-a\lambda\nonumber\\
&= a^2-a\lambda+\frac{\lambda^2}{4}-\frac1{16}\left(\lambda^4 - 4a\lambda^3+4a^2\lambda^2 \right)\nonumber\\
&= a^2\left(1-\frac{\lambda^2}{4}\right)-a\lambda\left(1-\frac{\lambda^2}{4}\right)+\frac{\lambda^2}{4}\left(1 -\frac{\lambda^2}{4} \right)\nonumber\\
&= \left(1-\frac{\lambda^2}{4}\right)\left(a^2-a\lambda+\frac{\lambda^2}{4}\right)\nonumber\\
&= \left(1-\frac{\lambda^2}{4}\right)\left(a-\frac{\lambda}{2}\right)^2\label{e:ggam1}\\
&>\frac{\lambda^4}{4}\left(a-\frac{\lambda}{2}\right)^4,\label{e:ggam2}
\end{align} 
where the last inequality used $1-\lambda^2/4 > \frac{\lambda^2}{4}(a-\lambda/2)^2$, which is equivalent to the first assumption of Case III.

We know that $\abs \lambda < 2$ by assumption. Therefore, $1-\lambda^2/4$ is positive, and thus, the minimum of $g(\lambda,a,\gamma_0(a,\lambda))$ will occur when $a$ is as close as possible to $\lambda/2$.  We claim that, given the Case III assumptions, we have $\abs{a -\lambda/2} \ge c_2:= -1+\sqrt{1+\frac12\abs{a^2+\gamma -1}}>0$, a constant depending only on $a$ and $\gamma$.  To prove the claim, assume for a contradiction that $ \abs{a -\lambda/2} < c_2$.  This implies that $\abs a < \abs\lambda/2+ c_2$, and so we have
\begin{align*}
\abs{a^2+\gamma -1} &\le \abs{ \left(\frac{\abs \lambda}{2} + c_2\right)^2 + \gamma -1 }
= \abs{ \frac{\abs\lambda^2}4 + c_2 \abs \lambda+ c_2^2+\gamma-1}
\\
&\le \abs{ \frac{\abs\lambda^2}4+\gamma-1} + \abs{c_2 \abs \lambda+ c_2^2}\\
&<\frac12\abs{a^2+\gamma -1}+2c_2+c_2^2,
\end{align*}
where the last inequality used the assumptions of Case III.  The chain of inequalities above implies that
$\frac12\abs{a^2+\gamma -1} +1 < (1+c_2)^2$, which is a contradiction when combined with the definition of $c_2$.  This proves the claim that $\abs{a -\lambda/2} \ge c_2$.  

To complete the proof of Case III, it is sufficient to show that either $(1-\lambda^2/4)c_2^2$  (from \eqref{e:ggam1}) or $\lambda^4c_2^4/4$ (from \eqref{e:ggam2}) is bounded from below.  This is easily done:  if $\abs\lambda < 1/2$, then $(1-\lambda^2/4)c_2^2 > \frac{15}{16}c_2^2$; and if $\abs{\lambda} \ge 1/2$, then $\lambda^4c_2^4/4> c_2^4/64$.  In both cases, the lower bound is a constant depending only on $a$ and $\gamma$, which completes the proof of Case III and hence also the proof of Lemma~\ref{lem:sing_min}.

\medskip

\noindent{\bf Acknowledgment:} The first author would like to thank Zhigang Bao for useful discussions. The authors would like to thank the anonymous referees for their valuable suggestions that have enhanced the readability of this paper.



\begin{thebibliography}{KMM{\etalchar{+}}13}

\bibitem[ADK22]{ADK22}
Johannes Alt, Raphael Ducatez, and Antti Knowles.
\newblock The completely delocalized region of the Erd{\H{o}}s-R{\'e}nyi graph.
\newblock {\em Electronic Communications in Probability},
  27, 2022.

\bibitem[AFH15]{AFH}
Omer Angel, Joel Friedman, and Shlomo Hoory.
\newblock The non-backtracking spectrum of the universal cover of a graph.
\newblock {\em Transactions of the American Mathematical Society},
  367(6):4287--4318, 2015.

\bibitem[Arn67]{Arnold1967}
Ludwig Arnold.
\newblock On the asymptotic distribution of the eigenvalues of random matrices.
\newblock {\em J. Math. Anal. Appl.}, 20:262--268, 1967.

\bibitem[Bas92]{Bass}
Hyman Bass.
\newblock The ihara-selberg zeta function of a tree lattice.
\newblock {\em International Journal of Mathematics}, 3(06):717--797, 1992.

\bibitem[BGBK17]{BBK17}
Florent Benaych-Georges, Charles Bordenave, and Antti Knowles.
\newblock Spectral radii of sparse random matrices.
\newblock {\em  Ann. Inst. Henri Poincar\'e Probab. Stat.}, 56(3):2141--2161, 2020.

\bibitem[BGK16]{BK16}
Florent Benaych-Georges and Antti Knowles.
\newblock Lectures on the local semicircle law for Wigner matrices.
\newblock {\em arXiv preprint arXiv:1601.04055}, 2016.

\bibitem[Bor15]{B15} 
Charles Bordenave.
\newblock A new proof of Friedman's second eigenvalue Theorem and its extension to random lifts.
\newblock {\em Annales Scientifiques de l École Normale Supérieure} 53(6), 2015.


\bibitem[BLM15]{BLM}
Charles Bordenave, Marc Lelarge, and Laurent Massouli{\'e}.
\newblock Non-backtracking spectrum of random graphs: community detection and
  non-regular ramanujan graphs.
\newblock {\em  Ann. Probab.}, 46(1): 1--71, 2018.

\bibitem[Cha09]{Cha09}
Djalil Chafa{\" i}.
\newblock Singular values of random matrices.
\newblock {\em Lecture Notes}, 2009.

\bibitem[Che52]{Chernoff52}
Herman Chernoff.
\newblock A measure of asymptotic efficiency for tests of a hypothesis based on
  the sum of observations.
\newblock {\em Ann. Math. Statistics}, 23:493--507, 1952.

\bibitem[CZ21]{CZ21}
Simon Coste and Yizhe Zhu.
\newblock Eigenvalues of the non-backtracking operator detached from the bulk.
\newblock {\em Random Matrices: Theory and Applications}, 10(03):2150028, 2021.


\bibitem[EKYY13]{EKYY}
L\'aszl\'o Erd{\H o}s, Antti Knowles, Horng-Tzer Yau, and Jun Yin.
\newblock Spectral statistics of {E}rd{\H o}s-{R}\'enyi graphs {I}: {L}ocal
  semicircle law.
\newblock {\em Ann. Probab.}, 41(3B):2279--2375, 2013.

\bibitem[FK81]{FK1981}
Z.~F\"uredi and J.~Koml\'os.
\newblock The eigenvalues of random symmetric matrices.
\newblock {\em Combinatorica}, 1(3):233--241, 1981.

\bibitem[HKM19]{HKM19}
Yukun He, Antti Knowles, and Matteo Marcozzi.
\newblock Local law and complete eigenvector delocalization for supercritical Erd{\H{o}}s--R{\'e}nyi graphs.
\newblock {\em Ann. Probab.}, 47(5):3278--3302, 2019.

\bibitem[GLM16]{GLM}
Lennart Gulikers, Marc Lelarge, and Laurent Massouli{\'e}.
\newblock Non-backtracking spectrum of degree-corrected stochastic block
  models.
\newblock {\em ITCS 2017-8th Innovations in Theoretical Computer Science}, 1--52, 2017.

\bibitem[Gre63]{Grenader1963}
Ulf Grenander.
\newblock {\em Probabilities on algebraic structures}.
\newblock John Wiley \& Sons, Inc., New York-London; Almqvist \& Wiksell,
  Stockholm-G\"oteborg-Uppsala, 1963.

\bibitem[Has89a]{H89}
Ki-ichiro Hashimoto.
\newblock Zeta functions of finite graphs and representations of {$p$}-adic
  groups.
\newblock In {\em Automorphic forms and geometry of arithmetic varieties},
  volume~15 of {\em Adv. Stud. Pure Math.}, pages 211--280. Academic Press,
  Boston, MA, 1989.

\bibitem[Has89b]{H}
Ki-ichiro Hashimoto.
\newblock Zeta functions of finite graphs and representations of p-adic groups.
\newblock {\em Automorphic forms and geometry of arithmetic varieties.}, pages
  211--280, 1989.
  
\bibitem[HK21]{HK21}
Yukun He and Antti Knowles.
\newblock Fluctuations of extreme eigenvalues of sparse Erd{\H{o}}s--R{\'e}nyi graphs.
\newblock {\em Probability Theory and Related Fields}, 1--72, 2021.

\bibitem[HLL83]{HLL1983}
Paul~W. Holland, Kathryn~Blackmond Laskey, and Samuel Leinhardt.
\newblock Stochastic blockmodels: first steps.
\newblock {\em Social Networks}, 5(2):109--137, 1983.

\bibitem[Hoe63]{Hoeffding1963}
Wassily Hoeffding.
\newblock Probability inequalities for sums of bounded random variables.
\newblock {\em J. Amer. Statist. Assoc.}, 58:13--30, 1963.

\bibitem[HLY20]{HLY20}
Jiaoyang Huang and Benjamin Landon and Horng-Tzer Yau.
\newblock Transition from Tracy--Widom to Gaussian fluctuations of extremal eigenvalues of sparse Erd{\H{o}}s--R{\'e}nyi graphs.
\newblock {\em The Annals of Probability}, 48(2):916--962, 2020.

\bibitem[Ivc73]{Ivchenko73}
Grigorii~Ivanovich Ivchenko.
\newblock On the asymptotic behavior of degrees of vertices in a random graph.
\newblock {\em Theory of Probability \& Its Applications}, 18(1):188--195,
  1973.

\bibitem[KMM{\etalchar{+}}13]{Krz13}
Florent Krzakala, Cristopher Moore, Elchanan Mossel, Joe Neeman, Allan Sly,
  Lenka Zdeborov{\'a}, and Pan Zhang.
\newblock Spectral redemption in clustering sparse networks.
\newblock {\em Proc. Natl. Acad. Sci. USA}, 110(52):20935--20940, 2013.

\bibitem[KN11]{KN2011}
Brian Karrer and M.~E.~J. Newman.
\newblock Stochastic blockmodels and community structure in networks.
\newblock {\em Phys. Rev. E (3)}, 83(1):016107, 10, 2011.

\bibitem[KP93]{KP93}
A.~M. Khorunzhy and L.~A. Pastur.
\newblock Limits of infinite interaction radius, dimensionality and the number
  of components for random operators with off-diagonal randomness.
\newblock {\em Comm. Math. Phys.}, 153(3):605--646, 1993.

\bibitem[KS03]{KS03}
Michael Krivelevich and Benny Sudakov.
\newblock The largest eigenvalue of sparse random graphs.
\newblock {\em Combin. Probab. Comput.}, 12(1):61--72, 2003.

\bibitem[LS18]{LS18}
Ji Oon Lee and Kevin Schnelli.
\newblock Local law and Tracy--Widom limit for sparse random matrices.
\newblock {\em Probability Theory and Related Fields}, 171(1):543--616, 2018.

\bibitem[LvHY18]{LvHY18}
Lata{\l a}, Rafa{\l } and van Handel, Ramon and Youssef, Pierre.
\newblock The dimension-free structure of nonhomogeneous random matrices.
\newblock {\em Inventiones mathematicae}, 214(3):1031--1080, 2018.

\bibitem[Mas14]{M2014}
Laurent Massouli{\'e}.
\newblock Community detection thresholds and the weak {R}amanujan property.
\newblock In {\em S{TOC}'14---{P}roceedings of the 2014 {ACM} {S}ymposium on
  {T}heory of {C}omputing}, pages 694--703. ACM, New York, 2014.

\bibitem[MNS13]{MNS13}
Elchanan Mossel, Joe Neeman, and Allan Sly.
\newblock A proof of the block model threshold conjecture.
\newblock {\em Combinatorica}, 38(3):665--708, 2018.

\bibitem[MNS15]{MNS15}
Elchanan Mossel, Joe Neeman, and Allan Sly.
\newblock Reconstruction and estimation in the planted partition model.
\newblock {\em Probab. Theory Related Fields}, 162(3-4):431--461, 2015.

\bibitem[SM20]{SM20}
Ludovic Stephan and Laurent Massouli{\'e}.
\newblock Non-backtracking spectra of weighted inhomogeneous random graphs.
\newblock {\em Mathematical Statistics and Learning}, 5(3):201--271, 2022.

\bibitem[Sil00]{Silvester}
John~R Silvester.
\newblock Determinants of block matrices.
\newblock {\em The Mathematical Gazette}, 84(501):460--467, 2000.

\bibitem[TV10]{TaoVu2010}
Terence Tao and Van Vu.
\newblock Random matrices: universality of {ESD}s and the circular law.
\newblock {\em Ann. Probab.}, 38(5):2023--2065, 2010.
\newblock With an appendix by Manjunath Krishnapur.

\bibitem[TVW13]{TVW13}
Linh~V. Tran, Van~H. Vu, and Ke~Wang.
\newblock Sparse random graphs: eigenvalues and eigenvectors.
\newblock {\em Random Structures Algorithms}, 42(1):110--134, 2013.

\bibitem[Vu07]{Vu07}
Van Vu.
\newblock {\em Spectral norm of random matrices}.
\newblock {\em Combinatorica}, 27(6):721--736, 2007.

\bibitem[Vu08]{Vu2008}
Van Vu.
\newblock {\em Random Discrete Matrices}, pages 257--280.
\newblock Springer Berlin Heidelberg, Berlin, Heidelberg, 2008.

\bibitem[Wig55]{Wigner1955}
Eugene~P. Wigner.
\newblock Characteristic vectors of bordered matrices with infinite dimensions.
\newblock {\em Ann. of Math.}, (2) 62:548--564, 1955.

\bibitem[Wig58]{Wigner1958}
Eugene~P. Wigner.
\newblock On the distribution of the roots of certain symmetric matrices.
\newblock {\em Ann. of Math.}, (2) 67:325--327, 1958.

\bibitem[Woo12]{Wood2012}
Philip~Matchett Wood.
\newblock Universality and the circular law for sparse random matrices.
\newblock {\em Ann. Appl. Probab.}, 22(3):1266--1300, 2012.

\bibitem[Woo16]{Wood2016}
Philip~Matchett Wood.
\newblock Universality of the {ESD} for a fixed matrix plus small random noise:
  a stability approach.
\newblock {\em Ann. Inst. Henri Poincar\'e Probab. Stat.}, 52(4):1877--1896,
  2016.

\end{thebibliography}




\newcommand{\etalchar}[1]{$^{#1}$}


\end{document}